\documentclass[11pt]{amsart}
\usepackage{amssymb, latexsym, amsmath, amsfonts, hyperref, enumitem, fancyhdr}
\usepackage{cleveref}

\newtheorem{thm}{Theorem}[section]

\newtheorem{lem}[thm]{Lemma}

\theoremstyle{definition}
\newtheorem{defn}[thm]{Definition}
\theoremstyle{remark}

\numberwithin{equation}{section}

\hypersetup{
  colorlinks,
  citecolor=black,
  linkcolor=black,
  urlcolor=blue}
\theoremstyle{plain}

  {\end{enumerate}}

\hypersetup{
  colorlinks,
  citecolor=blue,
  linkcolor=blue,
  urlcolor=blue}

\begin{document}
\title{CONTRACTIVITY VS. COMPLETE CONTRACTIVITY VIA PROPERTY P}

\thanks{The named author is supported by Council for Scientific and Industrial Research, MHRD, Government of India.}
%
\author{Samya Kumar Ray}
%

\address{Samya Kumar Ray: Department of Mathematics and Statistics, Indian Institute of Technology, Kanpur-208016}
\email{samya@math.iitk.ac.in}

\pagestyle{headings}

\begin{abstract}In this paper, we consider the question of contractivity vs. complete contractivity for domains in $\mathbb{C}^2$, which are unit balls with respect to some norm. We show that for a large class of Reinhardt domains, the corresponding Banach spaces do not have Property P, which implies that there exists contractive homomorphisms on these domains which are not completely contractive. At the end, we present a simple proof of the fact that the complex Banach spaces $(\mathbb{C}^2,\|\cdot\|_{\infty})$ and $(\mathbb{C}^3,\|\cdot\|_{\infty})$ have Property P.
\end{abstract}

\maketitle
\section{Introduction}
Let $\Omega\subseteq\mathbb{C}^m$ be a bounded domain, and $\mathbb{C}^{p\times q}$ be the Banach space of $p\times q$ complex matrices endowed with the operator norm. Let $\mathcal{O}(\Omega)$ denote the algebra of functions, each of which is holomorphic on some open set containing the closed set $\overline{\Omega}$. For any $w\in\Omega,$ and matrices $A_1,\dots,A_m$ in $\mathbb{C}^{p\times q}$, define $\langle\Delta f(w),\textbf{A}\rangle:=\:\sum_{j=1}^mA_j\frac{\partial f}{\partial z_j}(w),$ where $ f\in\mathcal{O}(\Omega)$ and $\textbf{A}:=(A_1,...,A_m)$. The map defined as
\begin{equation*} 
\Phi_{(w,\textbf{A})}(f):=
\begin{pmatrix}
f(w)I_p & \langle\Delta f(w),\textbf{A}\rangle\\
0 & f(w)I_q\\
\end{pmatrix},
\end{equation*} for all $f$ in $\mathcal{O}(\Omega),$ is clearly a continuous unital algebra homomorphism from $(\mathcal{O}(\Omega),\| \cdot \|_{\infty})$ to $(\mathbb{C}^{p+q}, \| \cdot \|_{op})$. We call such a homomorphism a Parrot Like Homomorphism. Now given a Parrot Like Homomorphism, we associate a natural map $$\Phi_{(w,\textbf{A})}\otimes I_k:\mathcal{O}(\Omega)\otimes\mathbb{C}^{k}\rightarrow (\mathbb{C}^{pk\times qk},\|\cdot\|_{op}),$$ where for $F\in\mathcal{O}(\Omega)\otimes\mathbb{C}^{k}$, the norm is defined as $\|F\|:=\sup\{\|F(z)\|_{op}:z\in\Omega\}$. The map $\Phi_{(w,\textbf{A})}$ is called contractive if $\|\Phi_{(w,\textbf{A})}\|_{op}\leq 1$ and is called completely contractive if $\sup_k{\|\Phi_{(w,\textbf{A})}\otimes I_k\|_{op}}\leq 1$. It is an open problem (see \cite{PV1}, \cite{PV2}) to determine domains for which there exists a contractive homomorphism which is not completely contractive.

In this paper, we answer this question for domains of the form $\{(z_1,z_2):|z_1|^p+|z_2|^q<1\}$, where $p,q$ are bigger than or equal to one and one of them is strictly bigger than one. At the end, we produce a simple proof of the fact that the Banach spaces $(\mathbb{C}^2,\|\cdot\|_{\infty})$ and $(\mathbb{C}^3,\|\cdot\|_{\infty})$ have Property P.

Our main theorem is the following.
\begin{thm}\label{0} Suppose $\Omega$ is a unit ball with respect to some norm in $\mathbb{C}^2$ which is Reinhardt, then $(\mathbb{C}^2,\|\cdot\|_{\Omega})$ has Property P if and only if $(\mathbb{R}^2,\|\cdot\|_{|\Omega|})$ has Property P.
\end{thm}
Given a complex Banach space $(\mathbb{C}^2,\|\cdot\|_{\Omega})$, where the unit ball $\Omega$ is Reinhardt, we define a corresponding real two dimensional Banach space $(\mathbb{R}^2,\|\cdot\|_{|\Omega|})$ with norm defined as $\|(x,y)\|_{|\Omega|}=\|(x,y)\|_{\Omega},$ for $(x,y)\in\mathbb{R}^2$. The unit ball of $(\mathbb{R}^2,\|\cdot\|_{|\Omega|})$ is denoted by $|\Omega|$.

For any two Banach spaces $E$, $F$ and a linear map $A:E\to F$, the operator norm of $A$ is denoted by $\|A\|_{E\to F}$. Often, we use $\|A\|_{op}$ to denote the operator norm, when the underlying Banach spaces on which $A$ acts are well understood.

For any complex number $z$, we denote the argument of $z$ by $\text{arg}z$. Similarly, for any real number $x$, we denote $\text{sgn}x$ to be the sign of $x$. 

The Banach space $\mathbb{C}^n$ with the usual supremum norm is denoted by $(\mathbb{C}^n,\|\cdot\|_{\infty})$. Also, by $(\mathbb{C}^n,\|\cdot\|_{1})$, we mean the Banach space with the norm given by
$\|(z_1,\dots,z_n)\|_1=\sum_{j=1}^n|z_j|$. We denote the corresponding real Banach spaces by similar notations.

The Banach space of all complex square summable sequences is denoted by $l_2$, which is known to be a Hilbert space equipped with the inner product $\langle \alpha,\beta\rangle=\sum_{i=1}^\infty\alpha_i\beta_i$, for all $\alpha=(\alpha_i)$ and $\beta=(\beta_i)$ in $l_2$.

Given any real (complex) matrix $A$, we write $A\geq 0$, whenever it is positive semidefinite. For a complex matrix $A\geq 0$, we define $|A|=\begin{pmatrix}
a_{11} & |a_{12}|\\
|a_{12}| & a_{22}\\
\end{pmatrix}$, where $A=\begin{pmatrix}
a_{11} & a_{12}\\
\overline{a_{12}} & a_{22}\\
\end{pmatrix}$. Clearly, $|A|\geq 0$ if and only if $A\geq 0$. Similarly, for a real matrix $A\geq 0$, we define $|A|=\begin{pmatrix}
a_{11} & |a_{12}|\\
|a_{12}| & a_{22}\\
\end{pmatrix}$, where $A=\begin{pmatrix}
a_{11} & a_{12}\\
a_{12} & a_{22}\\
\end{pmatrix}$.

Let $X$ be a Banach space, we associate a numerical constant $\gamma(X)$ as follows
\begin{equation}\label{-1}
\gamma(X):=\sup\{\langle A,B\rangle:A\geq O,B\geq O, \ \|A\|_{X\to X^*},\|B\|_{X^*\to X}\leq 1\},
\end{equation}
where the inner product in \eqref{-1} is the Hilbert-Schimdt inner product.
\begin{defn}[Property P]A Banach space $X$ is said to have Property P if and only if $\gamma(X)=1$.
\end{defn}

It has been proved in \cite{BG} (which was originally observed by Pisier), that Property P is actually equivalent to Two Summing Property. For more about Two Summing Property, the author recommends the reader \cite{AA} and \cite{PG3}.
\begin{defn}[Correlation Matrix]
A complex positive semidefinite matrix with all its diagonal elements equal to one is called a Correlation matrix. We denote the set of all $n\times n$ Correlation matrices by $\mathcal{C}(n)$.

From now on, we shall always assume that $\Omega$ is a Reinhardt domain in $\mathbb{C}^2$ which is a unit ball with respect to some norm.
\section{Main Results}We begin the proof of \Cref{0} by the following series of lemmas.
\end{defn}
\begin{lem}\label{1}
Let $A$ be a complex positive semidefinite matrix and $A:(\mathbb{C}^2,\|\cdot\|_{\Omega})\mapsto(\mathbb{C}^2,\|\cdot\|_{\Omega})^*$, then $\|A\|_{op}=\sup_{(z_1,z_2)\in\Omega}\sum_{i,j=1}^2a_{ij}z_i\bar{z}_{j}$.
\end{lem}
\begin{proof}
Applying the definition of the operator norm, one has 
\begin{align*}
\|A\|_{op}&= \ \ \sup_{X\in\Omega}\|AX\|_{{\Omega}^*}\\
&=\sup_{X\in\Omega,Y\in\Omega}|\langle AX,Y\rangle|.
\end{align*}
Using the fact that if $A\geq 0$, we can always finds a positive square root, say $B$ of $A$. Thus, we observe the following
 \begin{align}\label{f}
\sup_{X\in\Omega,Y\in\Omega}|\langle AX,Y\rangle|&=\sup_{X\in\Omega,Y\in\Omega}\langle BX,BY\rangle\\\nonumber
&\leq\sup_{X\in\Omega,Y\in\Omega}\|BX\|_2\|BY\|_2\\
&=\sup_{X\in\Omega}\|BX\|_2^2.\nonumber
\end{align}
In the last inequality, we have used Cauchy-Schwartz inequality. To complete the proof, we notice that one can always take $X=Y$ in \eqref{f}.
\end{proof}
\begin{lem}\label{2}
Let $(\mathbb{C}^2,\|\cdot\|_{\Omega})$ be a Banach space with the unit ball $\Omega,$ such that $\Omega$ is Reinhardt. Then, the unit ball of $(\mathbb{C}^2,\|\cdot\|_{\Omega})^*$ is again Reinhardt.
\end{lem}
\begin{proof} Given, $(z_1,z_2)$ in the dual unit ball of $(\mathbb{C}^2,\|\cdot\|_{\Omega})$, we observe the following
\begin{align*}
\|(z_1,z_2)\|_{\Omega^*}&=\sup_{(w_1,w_2)\in\Omega}|z_1\overline{w_1}+z_2\overline{w_2}|\\
&=\sup_{(w_1,w_2)\in\Omega}||z_1|\overline{e^{-i\text{arg}z_1}w_1}+|z_2|\overline{e^{-i\text{arg}z_2}w_2}|.
\end{align*}
Since $(e^{-i\text{arg}z_1}w_1,e^{-i\text{arg}z_2}w_2)\in\Omega$ as $\Omega$ is a Reinhardt domain, we get $$\|(z_1,z_2)\|_{\Omega^*}\leq\sup_{(w_1,w_2)\in\Omega}||z_1|\overline{w_1}+|z_2|\overline{w_2}|. $$ We pick an $\epsilon> 0$ and consider $(v_1,v_2)\in \Omega$ such that, $$||z_1|\overline{v_1}+|z_2|\overline{v_2}|\geq\sup_{(w_1,w_2)\in\Omega}  ||z_1|\overline{w_1}+|z_2|\overline{w_2}|-\epsilon.$$  Thus, we have $$|z_1\overline
{e^{i\text{arg}z_1}v_1}+z_2\overline
{e^{i\text{arg}z_2}v_2}|\geq\sup_{(w_1,w_2)\in\Omega}||z_1|\overline{w_1}+|z_2|\overline{w_2}|-\epsilon.$$ Since, $\Omega$ is Reinhardt, it immediately follows that 
$$\sup_{(w_1,w_2)\in\Omega}|z_1\overline{w_1}+z_2\overline{w_2}|\geq\sup_{(w_1,w_2)\in\Omega}||z_1|\overline{w_1}+|z_2|\overline{w_2}|-\epsilon.$$ 
Taking $\epsilon$ arbitraily close to zero, we get the desired result.
\end{proof}

\begin{lem}\label{3}
If $(\mathbb{C}^2,\|\cdot\|_{\Omega})$ is such that the unit ball $\Omega$ is Reinhardt then we have $\|(z_1,z_2)\|_{\Omega}=\|(|z_1|,|z_2|)\|_{|\Omega|}$.
\end{lem}
\begin{proof}
This clearly follows from \Cref{2}.
\end{proof}

\begin{lem}\label{4}
Let $A\geq 0$ be a complex matrix. Then \[ \|A\|_{((\mathbb{C}^2,\|\cdot\|_{\Omega})\rightarrow(\mathbb{C}^2,\|\cdot\|_{\Omega})^*)}=\||A|\|_{((\mathbb{C}^2,\|\cdot\|_{\Omega})\rightarrow(\mathbb{C}^2,\|\cdot\|_{\Omega})^*)}. \]
\end{lem}
\begin{proof}
Adopting \Cref{1}, one obtains that
\begin{align*}
\|A\|_{((\mathbb{C}^2,\|\cdot\|_{\Omega})\rightarrow(\mathbb{C}^2,\|\cdot\|_{\Omega})^*)}&=\sup_{(z_1,z_2)\in\Omega}(a_{11}|z_1|^2+2Re(a_{12}z_1\overline{z_2})+a_{22}|z_2|^2).
\end{align*}

Define the follwing two quanities:
 $$M_1=\sup_{(z_1,z_2)\in\Omega}(a_{11}|z_1|^2+2Re(a_{12}z_1\overline{z_2})+a_{22}|z_2|^2)$$ and $$M_2=\sup_{(z_1,z_2)\in\Omega}(a_{11}|z_1|^2+2|a_{12}||z_1||z_2|+a_{22}|z_2|^2).$$

Note that $M_1\leq M_2$ holds trivially by the use of triangle inequality. Consider $(w_1,w_2)\in\Omega$, so that, $a_{11}|w_1|^2+2|a_{12}||w_1||w_2|+a_{22}|w_2|^2\geq M_2-\epsilon.$ Since, $\Omega$ is Reinhardt, we have that $(e^{i\text{arg}a_{12}}|w_1|,|w_2|)\in\Omega$. Hence, we obtain $$M_1\geq a_{11}|  e^{i\text{arg}a_{12}}|w_1||^2+2Re(a_{12}e^{i\text{arg}a_{12}}|w_1||w_2|)+a_{22}|w_2|^2\geq M_2-\epsilon.$$ The result follows by taking, $\epsilon\rightarrow 0$.
\end{proof}

\begin{lem}\label{5}
Let $A\geq 0$ be a real matrix and $\Omega$ is Reinhardt. Then, 
\[ \|A\|_{((\mathbb{R}^2,\|\cdot\|_{|\Omega|})\rightarrow(\mathbb{R}^2,\|\cdot\|_{|\Omega|})^*)}=\||A|\|_{((\mathbb{R}^2,\|\cdot\|_{|\Omega|})\rightarrow(\mathbb{R}^2,\|\cdot\|_{|\Omega|})^*)}. \]
\end{lem}
\begin{proof}
As in \Cref{4}, we define the following two quantities $$M_1^\prime = \sup_{\|(x,y)\|_{|\Omega|}\leq 1}(a_{11}x^2+2a_{12}xy+a_{22}y^2)$$ and $$M_2^\prime=\sup_{\|(x,y)\|_{|\Omega|}\leq 1}(a_{11}x^2+2|a_{12}||x||y|+a_{22}y^2).$$ Clearly, one has $M_1^\prime\leq M_2^\prime$. Now, exactly like in the proof of \Cref{4}, consider $(x_0,y_0)\in|\Omega|$ so that $$a_{11}{x_0}^2+2|a_{12}||x_0||y_0|+a_{22}{y_0}^2\geq M_2^\prime-\epsilon,$$ where $\epsilon>0$ is fixed but arbitrary. Rest of the proof follows similarly as of \Cref{4}, owing to the fact that if $(x_0,y_0)\in|\Omega|$, then one automatically has $(\text{sgn}a_{12}|x_0|,|y_0|)\in{|\Omega|}$. 
\end{proof}
\begin{lem}\label{6}
Let $A\geq 0$ be a complex matrix. Then,\[ \||A|\|_{((\mathbb{C}^2,\|\cdot\|_{\Omega})\rightarrow(\mathbb{C}^2,\|\cdot\|_{\Omega})^*)}=\||A|\|_{((\mathbb{R}^2,\|\cdot\|_{|\Omega|})\rightarrow(\mathbb{R}^2,\|\cdot\|_{|\Omega|})^*)}. \]
\end{lem}
\begin{proof}
By \Cref{1}, \Cref{2}, \Cref{3}, \Cref{4} and \Cref{5}, we obtain the following
\begin{align*}
\||A|\|_{((\mathbb{C}^2,\|\cdot\|_{\Omega})\rightarrow(\mathbb{C}^2,\|\cdot\|_{\Omega})^*)}&=\sup_{(z_1,z_2)\in\Omega}(a_{11}|z_1|^2+2|a_{12}||z_1||z_2|+a_{22}|z_2|^2)\\
&=\sup_{(|z_1|,|z_2|)\in\Omega}(a_{11}|z_1|^2+2|a_{12}||z_1||z_2|+a_{22}|z_2|^2)\\
&=\sup_{\|(x,y)\|_{|\Omega|}\leq 1}(a_{11}x^2+2|a_{12}||x||y|+a_{22}y^2)\\
&=\||A|\|_{((\mathbb{R}^2,\|\cdot\|_{|\Omega|})\rightarrow(\mathbb{R}^2,\|\cdot\|_{|\Omega|})^*)}.
\end{align*}
\end{proof}
We now turn our attention to prove \cref{0}.
\begin{proof}[Proof of \Cref{0}]
We observe the following
\begin{align*}
\gamma (( \mathbb{C}^2, &  \|\cdot\|_{\Omega}))\\
 & =  \sup\{\langle A,B\rangle :\|A
\|_{((\mathbb{C}^2,\|\cdot\|_{\Omega})\rightarrow(\mathbb{C}^2,\|\cdot\|_{\Omega})^*)}\leq 1,\|B\|_{((\mathbb{C}^2,\|\cdot\|_{\Omega}^*)\rightarrow(\mathbb{C}^2,\|\cdot\|_{\Omega}))}\leq 1,A\geq 0,B\geq 0\}\\
&\leq\sup\{\langle |A|,|B|\rangle :\|A
\|_{((\mathbb{C}^2,\|\cdot\|_{\Omega})\rightarrow(\mathbb{C}^2,\|\cdot\|_{\Omega})^*)}\leq 1,\|B\|_{((\mathbb{C}^2,\|\|_{\Omega}^*)\rightarrow(\mathbb{C}^2,\|\cdot\|_{\Omega}))}\leq 1,A\geq 0,B\geq 0\}.
\end{align*}
The last inequality is just the triangle inequality and in view of \cref{4}, it is clear that the inequality above is actually an equality. By a very similar argument as above and using \cref{5}, we also have 
\begin{align*}
\gamma & ((\mathbb{R}^2,  \|\cdot\|_{|\Omega|})) \\
 &  =  \sup\{\langle |C|,|D|\rangle :\|C
\|_{((\mathbb{R}^2,\|\|_{|\Omega|})\rightarrow(\mathbb{R}^2,\|\cdot\|_{|\Omega|})^*)}\leq 1,\|D\|_{((\mathbb{R}^2,\|\cdot\|_{|\Omega|})^*\rightarrow(\mathbb{R}^2,\|\cdot\|_{|\Omega|}))}\leq 1,C\geq 0,D\geq 0\}.
\end{align*}
Now, using the above and \Cref{6}, one readily sees that $\gamma((\mathbb{R}^2,\|\cdot\|_{|\Omega|}))=\gamma((\mathbb{C}^2,\|\cdot\|_{\Omega}))$.
\end{proof}
As an application of the above theorem and by Theorem 3.3 of \cite{AA}, one can easily observe that if $\Omega$ is of the form $\{(z_1,z_2):|z_1|^p+|z_2|^q<1\}$, where $p$ and $q$ are real numbers bigger than or equal to one and at least one of them is strictly bigger than one, then $(\mathbb{C}^2,\|\cdot\|_{\Omega})$ cannot have Property P, as $\overline{|\Omega|}=\{(x,y):|x|^p+|y|^q\leq 1,x,y\in\mathbb{R}\}$ has more than four extreme points but the unit ball of the real Banach space $(\mathbb{R}^2,\|\cdot\|_{\infty})$ has exactly four extreme points. This also produces simple proofs of some of the theorems proved in \cite{BG} and \cite{MAC}.

\textbf{Remark:} 
There is also an open problem which is very closely related to our problem. That is to determine Banach spaces which have unique operator space structure (see \cite{PG2}). However, in view of the chracterization of Thullen for Reinhardt domains in $\mathbb{C}^2$, we get a very large class of Banach spaces, which can be endowed with different operator space structures. We also suggests the reader \cite{MAC} for some recent progress.

The following theorem has been proved in \cite{AA} and \cite{BG}. However, our proof is much simpler in nature. The following theorem is a part of authors Master's thesis \cite{RKS}.

\begin{thm}The complex Banach spaces $(\mathbb{C}^2,\|\cdot\|_{\infty})$ and $(\mathbb{C}^3,\|\cdot\|_{\infty})$ have Property P.
\begin{proof}
Given, $A$ a complex $n\times n$ positive semi-definite matrix, we define $$\beta(A):=\sup_{B\in\mathcal{C}(n)}\langle A,B\rangle.$$ 
Note that, $\beta(A)=\sup_{\|x_i\|_2=1,\|y_j\|_2=1}|\sum_{i,j=1}^na_{ij}\langle x_i,y_j \rangle|$. This is because, if we use the techniques of \Cref{1}, it follows that $$\sup_{\|x_i\|_2=1}|\sum_{i,j=1}^na_{ij}\langle x_i,x_j \rangle|=\sup_{\|x_i\|_2=1,\|y_j\|_2=1}|\sum_{i,j=1}^na_{ij}\langle x_i,y_j \rangle|$$ and given any correlation matrix $C$, one can find $l_2$ unit vectors $x_i$'s such that $(\langle x_i,x_j\rangle )=C$ and  vise versa. Now, observing that the quantity $\langle A,B\rangle$ is linear in $B$ and $\mathcal{C}(n)$ is a compact convex set, we conclude that $\beta(A)=\sup_{B\in E(\mathcal{C}(n))}\langle A,B\rangle,$ where $E(\mathcal{C}(n))$ is the set of all extreme points of $\mathcal{C}(n)$. Since, all the elements of $E(\mathcal{C}(n))$ have ranks less than or equal to $\sqrt{n}$ (\cite{LCB}), in case, when $n=2,3$, we conclude that extreme correlation matrices have rank one. Now, if the correlation matrix $(\langle x_i,x_j\rangle )$ is of rank $1$, then $x_i$'s have to be one dimensional unit vectors. So for $n=2,3$, we obtain the following $$\beta(A)=\sup_{B\in E(\mathcal{C}(n))}\langle A,B\rangle=\sup_{\mid z_i\mid=1}\sum_{i,j=1}^{n}a_{ij}z_i\bar{z}_{j}=\|{A} \|_{(\mathbb{C}^n,\|\cdot\|_{\infty})\to(\mathbb{C}^n,\|\cdot\|_{1})}.$$ This proves the theorem.
\end{proof}
\end{thm}
\textbf{Acknowledgement:} The author wants to express his sincere gratitude to his supervisor Prof. Parasar Mohanty, from whom he first came to know about the original problem, while learning the theory of Operator Spaces. He also thanks Prof. Gadadhar Misra for many helpful insights and Dr. Avijit Pal for patiently explaining his Ph.d thesis. In addition, he wants to convey his acknowledgement to Dr. Qaiser Jahan and Dr. Rajeev Gupta for many fruitful assistance.

\end{document}